\documentclass{cedram-aif}

\newcommand\Aff{{\mathchoice {\setbox0=\hbox{$\displaystyle\rm Aff$}}
{\hbox{$\textstyle\rm Aff$}}
{\hbox{$\scriptstyle\rm Aff$}}
{\hbox{$\scriptscriptstyle\rm Aff$}}}}

\newtheorem{question}{Question}

\def\PSL{\mathop{\rm PSL}\nolimits}
\def\Sz{\mathop{\rm Sz}\nolimits}

\def\on{\mathord{\mathrm{O}^{^{\textstyle\mathrm{,}}}\mathrm{N}}}

\def\mat#1{\mathord{\mathrm{M}_{#1}}}
\def\co#1{\mathord{\mathrm{Co}_{#1}}}
\def\fis#1{\mathord{\mathrm{Fi}_{#1}}}

\def\jan#1{\mathord{\mathrm{J}_{#1}}}

\def\baby{\mathord{\mathrm{B}}}

\def\monster{\mathord{\mathrm{M}}}

\def\mcl{\mathord{\mathrm{M}^{\scriptstyle\mathrm{c}}\mathrm{L}}}

\begin{document}

\title[Large Noether number]{Finite groups with large Noether number are almost cyclic}


\thanks{The research was partly supported by the National Research, Development and Innovation Office (NKFIH) Grant No.~K115799. The second and third authors were also funded by the National Research, Development and Innovation Office (NKFIH) Grant No.~ERC$\_$HU$\_$15 118286. Their work on the project leading to this application has received funding from the European Research Council (ERC) under the European Union's Horizon 2020 research and innovation programme (grant agreement No 741420). The second author received funding from ERC 648017 and was supported by the J\'anos Bolyai Research Scholarship of the Hungarian Academy of Sciences.}

\author{\firstname{P\'al} \lastname{Heged\H us}}
\address{Department of Mathematics\\
Central European University\\
N\'ador utca 9\\
H-1051 Budapest, Hungary}
\email{hegedusp@ceu.edu}

\author{\firstname{Attila} \lastname{Mar\'oti}}
\address{Alfr\'ed R\'enyi Institute of Mathematics\\
Hungarian Academy of Sciences\\
Re\'altanoda utca 13-15\\
H-1053, Budapest, Hungary}
\email{maroti.attila@renyi.mta.hu}

\author{\firstname{L\'aszl\'o} \lastname{Pyber}}
\address{Alfr\'ed R\'enyi Institute of Mathematics\\
Hungarian Academy of Sciences\\
Re\'altanoda utca 13-15\\
H-1053, Budapest, Hungary}
\email{pyber.laszlo@renyi.mta.hu}

\altkeywords{polynomes invariants, majorant de Noether, groupes simples de type de Lie}

\date{\today}

\begin{abstract}
Noether, Fleischmann and Fogarty proved that if the characteristic of the underlying field does not divide the order $|G|$ of a finite group $G$, then the polynomial invariants of $G$ are generated by polynomials of degrees at most $|G|$. Let $\beta(G)$ denote the largest indispensable degree in such generating sets. Cziszter and Domokos recently described finite groups $G$ with $|G|/\beta(G)$ at most $2$.
We prove an asymptotic extension of their result. Namely, $|G|/\beta(G)$ is bounded for a finite group $G$ if and only if $G$ has a characteristic cyclic subgroup of bounded index. In the course of the proof we obtain the following surprising result. If $S$ is a finite simple group of Lie type or a sporadic group then we have $\beta(S) \leq {|S|}^{39/40}$. We ask a number of questions motivated by our results.
\end{abstract}

\begin{altabstract}Noether, Fleischmann et Fogarty ont montr\'e que si le caract\'eristique du corps sous-jacent ne divise pas l'ordre $|G|$ d'un groupe fini, alors l'anneau de polynomes invariants de $G$ est engendr\'e par des polynomes de degr\'e au plus \'egal \`a $|G|$.
Notons par $\beta(G)$ le plus haut degr\'e indispensable en un tel syst\`eme de g\'en\'erateurs. Cziszter et Domokos ont r\'ecemment d\'ecrit les groupes finis $G$ tels que $|G|/\beta(G)$ est au plus \'egal \`a $2$.  Nous d\'emontrons une extension asymptotique de leur r\'esultat, \`a savoir que $|G|/\beta(G)$ est born\'e pour un groupe fini $G$ si et seulement s'il admet un sous-groupe caract\'eristique cyclique d'indice born\'e. Durant la d\'emonstration nous trouvons le r\'esultat surprenant suivant :
si $S$ est un groupe fini simple de type de Lie ou l'un des groupes sporadiques alors on a
$\beta(S) \leq {|S|}^{39/40}$. Nous posons \'egalament quelques questions motiv\'ees par nos r\'esultats.
\end{altabstract}

\maketitle

\section{Introduction}

Let $G$ be a finite group and $V$ an $FG$-module of finite dimension over a field $F$. By a classical theorem of Noether \cite{noether}, the algebra of polynomial invariants on $V$, denoted by $F[V]^G$, is finitely generated. Define $\beta(G, V)$ to be the smallest integer $d$ such that $F[V]^G$ is generated by elements of degrees at most $d$. In case the characteristic of $F$ does not divide $|G|$, the numbers $\beta(G,V)$ have a largest value as $V$ ranges over the finite dimensional $FG$-modules. This number is called the \emph{Noether number} and is denoted by $\beta(G)$. The notation $\beta(G)$ suppresses the dependence on the field but it should not cause misunderstanding. In fact, for fields of the same characteristic the Noether number is the same and we may assume that $F$ is algebraically closed. See \cite{knop} for details.

Noether \cite{noether} also proved that $\beta(G)\leq |G|$ over fields of characteristic $0$. This bound was verified independently by Fleischmann \cite{fleischmann} and Fogarty \cite{fogarty} to hold also in positive characteristics not dividing $|G|$. For characteristics dividing $|G|$, a deep result of Symonds \cite{symonds} states that $\beta(G,V)\leq\dim(V)(|G|-1)$.

From now on throughout the whole paper, except in Question \ref{mod}, we assume that the characteristic of the field $F$ is $0$ or is coprime to the order of $G$.

Schmid \cite{schmid} proved that over the field of complex numbers $\beta(G)=|G|$ holds only when $G$ is cyclic. This was sharpened by Domokos and Heged\H us \cite{DH} (and later by Sezer \cite{sezer} in positive coprime characteristic) to $\beta(G)\leq \frac{3}{4}|G|$ unless $G$ is cyclic.

An important ingredient in Schmid's argument was to show that $\beta(G)\geq\beta(H)$ holds for any subgroup $H\leq G$. In particular, $\beta(G)$ is bounded from below by the maximal order of the elements in $G$, that is, the \emph{Noether index} $n(G) = |G|/\beta(G)$ of a finite group $G$ is at most the minimal index of a cyclic subgroup in $G$.

Recently Cziszter and Domokos \cite{CzD} described finite groups $G$ with $n(G)$ at most $2$. Their deep result \cite[Theorem 1.1]{CzD} states that for a finite group $G$ (with order not divisible by the characteristic of $F$) we have $n(G) \leq 2$ if and only if $G$ has a cyclic subgroup of index at most $2$, or $G$ is isomorphic to $Z_3\times Z_3$, $Z_2\times Z_2\times Z_2$, the alternating group $A_4$, or the binary tetrahedral group $\widetilde{A_4}$. In particular, the inequality $n(G) \leq 2$ implies that $G$ has a cyclic subgroup of index at most $4$.

Our main result is as follows.

\begin{theo}
\label{mainmainindex}
Let $G$ be a finite group with Noether index $n(G)$. Then $G$ has a characteristic cyclic subgroup of index at most ${n(G)}^{10 \log_{2}k}$ where $k$ denotes the maximum of $2^{10}$ and the largest degree of a non-Abelian alternating composition factor of $G$, if such exists. Furthermore if $G$ is solvable, then $G$ has a characteristic cyclic subgroup of index at most ${n(G)}^{10}$.
\end{theo}

In view of Theorem \ref{main} and Section \ref{almostsimplesection}, the bound ${n(G)}^{10}$ holds even for a large class of non-solvable groups.

Theorem \ref{mainmainindex} has a consequence which can be viewed as an asymptotic version of the afore-mentioned result of Cziszter and Domokos.

\begin{coro}
\label{corollary}
Let $G$ be a finite group with Noether index $n(G)$. If $G$ is nonsol-vable, then $n(G)>2.7$ and $G$ has a characteristic cyclic subgroup of index at most $n(G)^{100 + 10 \log_{2} \log_{2} n(G)}$. If $G$ is solvable then $G$ contains a characteristic cyclic subgroup of index at most $n(G)^{10}$.
\end{coro}

It is an open question whether there exists a polynomial bound in $n(G)$ for the index of a characteristic cyclic subgroup in an arbitrary finite group $G$. Theorem \ref{mainmainindex} is a major step in answering this question.

As a step in our proofs we obtain a result which may be of independent interest.

\begin{theo}
\label{simple}
Let $S$ be a finite simple group of Lie type or a sporadic simple group. Then $n(S) \geq |S|^{1/40}$.
\end{theo}

It would be interesting to know if the bound in Theorem \ref{simple} holds for alternating groups of arbitrarily large degrees. Our methods are sufficient only for degrees up to $17$. For degrees no greater than $17$ (but at least $5$) the claim follows from the remark after Lemma \ref{czd_simple}.

Assume that, for some fixed constant $\epsilon > 0$, we have $n(S) \geq {|S|}^{\epsilon}$ for every alternating group $S$ of degree at least $5$. Then our proofs show that, for some other (computable) fixed constant $\epsilon' > 0$ with $\epsilon' \leq 0.1$, any finite group $G$ has a characteristic cyclic subgroup of index at most ${n(G)}^{1/\epsilon'}$.

\section{Affine groups}
\label{semisection}

Our main aim in the present section is to give upper bounds on $\beta(G)$ for the Frobenius group $G\cong Z_p\rtimes Z_n$, where $p$ is a prime and $n\mid p-1$.

It is an open conjecture of Pawale \cite{pawale} that $\beta(Z_p\rtimes Z_q)=p+q-1$ for a prime $q$. This is verified for $q=2$ \cite{schmid} (where $\beta(D_{2n})=n+1$ is shown for composite $n$, as well) and for $q=3$ \cite{cz3p}. Cziszter and Domokos obtain an upper bound which we extend to a more general one if $q$ is not a prime. See Lemma~\ref{q08}, Theorem~\ref{prootn} and Corollary~\ref{pointnine}.

In this section we rely heavily on the techniques developed by Cziszter and Domokos. For convenience and completeness we include here those that we need. However, we try to simplify and not include them in full generality.

Let $G$ be the Frobenius group of order $pn$ with $Z_p\le G\le \Aff_p$. Then every $G$-module has a $Z_p$-eigenbasis permuted up to scalars by $G$.  The regular module is relevant because every irreducible $Z_p$-character occurs in it. For every $Z_p$-module $V$ the polynomial invariants are linear combinations of $Z_p$-invariant monomials. The $Z_p$-invariant monomials correspond to $0$-sum sequences of irreducible $Z_p$-characters. These motivate all the definitions below.

Let  ${\mathcal Y}=\{y_1,\ldots,\, y_p\}$ be the set of variables from $F[Z_p]$ that are $Z_p$-eigenvectors and $y_1$ is $Z_p$-invariant. For a monomial $f=\prod_{i=1}^p y_i^{a_i}$ let us define $b(f)=\prod_{a_i>0}y_i$. Let $g_1=b(f)$ and construct recursively the finite list of monomials $g_1,\,g_2,\ldots$ in such a way that $g_{k+1}=b(f/\prod_{j=1}^k g_j)$ for every $k$, stopping if $f=\prod g_j$.  We call this list the \emph{row decomposition} of $f$. (In \cite{CzD} the corresponding list of irreducible $Z_p$-characters is considered and called the row decomposition.) This list consists of monomials each dividing the previous one and the exponent of every variable $y_i$ is at most $1$.

Let $l$ be a positive integer. Suppose a set of variables $\{x_1,\ldots,\,x_l\}$ consists of $Z_p$-eigenvectors on which $G/Z_p$ acts by permutation, but not necessarily transitively. 
For each $x_i$ there is a corresponding unique $y_{\bar{i}}\in\mathcal Y$ having the same $Z_p$-action on them.
This defines a map $m\mapsto f_m$ from the monomials in $\{x_1,\ldots,\,x_l\}$ to the monomials in $\mathcal Y$ by $m=\prod x_i^{a_i}\mapsto m_f=\prod y_{\bar i}^{a_i}$. This map is $G/Z_p$-equivariant. Moreover, the $Z_p$-action on $m$ is the same as on $f_m$, so $m$ is $Z_p$-invariant if and only if $f_m$ is.

Given a monomial $m$ we determine the row decomposition $g_1,\ldots,\,g_h$ of $f_m$. Suppose that for every $G$-orbit ${\mathcal O}\subseteq\mathcal Y$ and every index $i<h$ the following holds. If $g_i$ involves some variables from  $\mathcal O$, but not all then $g_{i+1}$ involves fewer variables than $g_i$ does. Such a monomial $m$ is called \emph{gapless} in \cite[Definition~2.5]{CzD}. If $g_i=g_{i+1}$ for a gapless monomial $m$ then $g_i$ is $G/Z_p$-invariant. In particular, as nontrivial $G/Z_p$-orbits on $\mathcal Y$ are of length $n$,
\begin{equation}\text{if }y_1\nmid g_i\text{ and }\deg(g_i)<n\text{ then }\deg(g_{i+1})<\deg(g_i).\label{below_n}
\end{equation}

Let $M =\oplus_{d=0}^\infty M_d$ be a graded module over a commutative graded $F$-algebra $R=\oplus_{d=0}^\infty R_d$. We also assume that $R_0 = F$ when $1\in R$ and
$R_0 = {0}$ otherwise. 
Define $M_{\leq s} =\oplus_{d=0}^s M_d$, a subspace of $M$, and $R_+ =\oplus_{d=1}^\infty R_d\triangleleft R$ a maximal ideal. The subalgebra of $R$ generated by $R_{\leq s}$  is denoted by $F[R_{\leq s}]$. Define $\beta(M, R)=\min\{s\mid M=\langle M_{\leq s}\rangle_{R_+}\}$, the highest degree needed for an $R_+$-generating set of $M$. In other words, it is the highest degree of nonzero components of $M/MR_+$ (the factor space $M/MR_+$ inherits the grading).

The following three propositions from \cite{CzD} will be used in the proof of Theorem~\ref{prootn}. They are paraphrased and not stated in their full generality.
\begin{prop}\cite[Proposition 2.7]{CzD}\label{cd3.7}
Let $G$ be the Frobenius group of order $pn$ with $Z_p\le G\le \Aff_p$. Let $V$ be an $FG$-module, $L=F[V]$ the polynomial algebra, $R=L^G$ its invariants. Suppose the variables of $L$ are permuted by $G$ up to non-zero scalar multiples. Then the vector space $L_+/L_+R_+$ is spanned by monomials of the form $b_1\cdots b_r m$, where the $b_i$ are $Z_p$-invariant of degree $1$ or of prime degree $q_i|n$ and $m$ has a gapless divisor of degree at least $\deg(m)-(p-1)$.
\end{prop}

(Note that the so-called \emph{bricks} mentioned in the original version of Proposition~\ref{cd3.7} are $Z_p$-invariant.)

\begin{prop}\cite[Lemma 1.11]{CzD}
\label{cd2.31}
Let $G$ be the Frobenius group of order $pn$ with $Z_p\le G\le \Aff_p$. Let $V$ be an $FG$-module, $L=F[V]$ the polynomial algebra, $R=L^G$ and $I=L^{Z_p}$ its invariants. Then for every $s\ge 1$ the following bound is valid:
\[\beta(L_+,R)\le (n-1)s+\max\{\beta(L_+/L_+R_+,I),\beta(L_+/L_+R_+,F[I_{\leq s}])-s\}.
\]
\end{prop}

(The original version of Proposition~\ref{cd2.31} 
holds for the generalized Noether numbers $\beta_r$, however we only use the case $r=1$.)

\begin{lemm}\cite[Lemma 2.10]{CzD} Let $S$ be a sequence over $Z_p$ with maximal multiplicity $h$. If $|S|\geq p$ then $S$ has a zero-sum subsequence $T\subseteq S$ of length $|T|\leq h$.
\end{lemm}

The following proposition is a simple corollary.

\begin{prop}
\label{cd1.13}
Suppose $f$ is a monomial in $\mathcal Y$ of degree at least $p$ such that the exponent of each $y_i\in\mathcal Y$ is at most $h$. Then $f$ has a $Z_p$-invariant submonomial $f^\prime$ such that $\deg(f^\prime)\leq h$.
\end{prop}

\begin{proof}
Let $f=\prod y_i^{a_i}$. Fix a generator element $z\in Z_p$ and a primitive $p$-th root of unity, $\mu\in F$. Define $S$ to be the sequence over $\mathbb{Z}/\mathbb{Z}_p$ consisting of $a_i$ copies of the exponent of $\mu$ as the eigenvalue of $z$ on $y_i$. This satisfies the assumptions of the previous lemma. Let then $f^\prime$ be the product of the elements of $T$, it is a submonomial of degree $|T|\leq h$. That $T$ is zero-sum means exactly that $f^\prime$ is $Z_p$-invariant.
\end{proof}

The following upper bound is used frequently.

\begin{lemm}
\label{Olson}
Let $E = {(Z_{p})}^k$ be a non-cyclic elementary Abelian $p$-group for some prime $p$. Then $\beta(E) = kp - k + 1$. Thus $\beta(E) < {|E|}^{0.8}$. Furthermore if $|E| \not= 2^2$, $3^2$, $5^2$, then $\beta(E) < {|E|}^{0.67}$.
\end{lemm}

\begin{proof}
The first statement is the combination of Olson's Theorem \cite{olson} and a ``folklore result'' of invariant theory \cite[Proposition~8]{sezer}. We have $\beta(E) < {|E|}^{0.8}$ since $k \geq 2$. The other statement follows from an easy calculation.
\end{proof}

We reformulate the result of \cite{CzD} for affine groups in a form that can be applied in inductive arguments. For our purposes the following lemma is sufficient. However, as the proof shows, $\beta(G)\leq (1+\varepsilon)p\sqrt{q}$ is true for fixed $\varepsilon > 0$ and for $p$, $q$ large enough.

\begin{lemm}
\label{q08}
Let $q\mid p-1$ for primes $p,q$ and let $G\le \Aff_p$ be of order $pq$. Then $\beta(G)\leq pq^{0.8}$.
\end{lemm}

\begin{proof}
If $q=2$, then $\beta(G)=p+1<p2^{0.8}$ (see \cite[(7.1)]{schmid} and \cite[Proposition~13]{sezer}). Let $q>2$. By \cite[Proposition~2.15]{CzD} we have $\beta(G)\leq \frac{3}{2}(p+q(q-2))-2< 3p-2$ if $p>q(q-2)$. If here $q\geq 5$ then $3p-2<pq^{0.8}$. If $q=3$ then $\beta(G)$ is at most $\frac{3}{2}(p+q(q-2))-2=\frac{3}{2}p+2.5<p3^{0.8}$, as required.

So let $p<q(q-2)$, in particular $q>3$. In this case \cite[Proposition~2.15]{CzD} concludes $\beta(G)\leq 2p+(q-2)q-2$ and $\beta(G)\leq 2p+(q-2)(c-1)-2$ if there exists $c\leq q $ such that $c(c-1)<2p<c(c+1)$. Note that if $q(q-1)<2p$ then $q<\sqrt{2p}$ and if $q(q-1)>2p$ then there exists $c\leq q $ such that $c(c-1)<2p<c(c+1)$ and $c-1<\sqrt{2p}$. So in both cases $\beta(G)\leq 2p+(q-2)\sqrt{2p}-2$. If $q=5$ then $p<15$ and $5\mid p-1$ imply $p=11$. We have $\beta(G)\leq 22+3\sqrt{22}-2<11\cdot5^{0.8}$.

Finally let $q\geq 7$. Using $q-2<\sqrt{q}\sqrt{p/2}$ we get
\[\beta(G)< p(2+\frac{\sqrt{q}\sqrt{p/2}\sqrt{2p}}{p})=p(2+\sqrt{q}).
\]

As $q^{0.8}-q^{0.5}$ is increasing and $7^{0.8}-7^{0.5}>2$ we get the claimed bound.
\end{proof}

\begin{theo}
\label{prootn}
Let $G$ be the Frobenius group of order $pn$ with $Z_p\le G\le \Aff_p$. Suppose that $n \geq 6$ has no prime divisor larger than $p/\sqrt{n}$. Then $\beta(G)< 2p\sqrt{n}$.
\end{theo}

\begin{proof}
Let $V$  be an arbitrary $FG$-module, $L=F[V]$ the polynomial algebra and $R=L^G$ and $I=L^{Z_p}$ the respective group invariants. Put $s=[p/\sqrt{n}]$. As $\beta(Z_p)=p$ we have $\beta(L_+/L_+R_+,I)\leq p$. Hence by Proposition~\ref{cd2.31},
\[\beta(G,V)\le (n-1)s+\max\{p,\beta(L_+/L_+R_+,F[I_{\leq s}])-s\}.\]
The first term of this sum is smaller than $p\sqrt{n}$ so it is enough to prove that
\begin{equation}\beta(L_+/L_+R_+,F[I_{\leq s}])\le p\sqrt{n}+s.\label{p_sqrt_n+s}
\end{equation}
We assume that the basis of the dual module $V^*$ is a $Z_p$-eigenbasis $\{x_1,x_2,\ldots,x_l\}$ permuted by $G/Z_p$. 
Now apply Proposition~\ref{cd3.7}. The space $L_+/L_+R_+$ is spanned by monomials $m$ that either have a $Z_p$-invariant divisor of degree at most $s$ or have a gapless monomial divisor of degree at least $\deg(m)-(p-1)$. The former kind are in $F[I_{\leq s}]$ so we need an upper bound for the degrees of the latter kind. More precisely, we have that if $m^\prime$ is the largest degree gapless monomial with no $Z_p$-invariant divisor of degree at most $s$ then
\begin{equation}\label{L+R+_bound}
  \beta(L_+/L_+R_+,F[I_{\leq s}])\le p-1+\deg(m^\prime).
\end{equation}

Consider now the row decomposition $g_1,\ldots,\, g_h$ of $f_{m^{\prime}}$. In the submonomial $f=g_1+g_2+\cdots+g_s$ of $f_{m^\prime}$ all the exponents are at most $s$, so by Proposition~\ref{cd1.13}, $\deg f\leq p-1$. This implies that $\deg(g_{s})\leq (p-1)/s$. It is below $\sqrt{n}+1$ because if $s= (p/\sqrt{n})-\varepsilon$ then
\[
(\frac{p}{\sqrt{n}}-\varepsilon)(\sqrt{n}+1)=p+\frac{p}{\sqrt{n}}-\varepsilon\sqrt{n}-\varepsilon>p-1.
\]

So $\deg(g_s)\leq\sqrt{n}+1$. In particular, $\deg(g_s)<n$ and by \eqref{below_n}, $\deg(g_{i+1})<\deg(g_i)$ for $i\geq s$. Hence we have the following bound on the degree.
\[\deg(m^\prime)=\sum_{i=1}^s\deg(g_i)+\sum_{i=s+1}^h \deg(g_i)<p-1+\frac{1}{2}\sqrt{n}(\sqrt{n}+1)=p-1+\frac{n+\sqrt{n}}{2}.\]

Now \eqref{L+R+_bound} and  $2+\frac{n}{2(p-1)}\leq 2.5<\sqrt{n}+\frac{1}{\sqrt{n}}$ (as $n>5$) imply that
\begin{eqnarray*}
\beta(L_+/L_+R_+,F[I_{\leq s}])&\le&p-1+\deg(m^\prime)\leq 2(p-1)+\frac{n+\sqrt{n}}{2} =\\
&=& (p-1)\left(2+\frac{n}{2(p-1)}\right)+\frac{\sqrt{n}}{2}<\\
 &<& (p-1)\left(\sqrt{n}+\frac{1}{\sqrt{n}}\right)+\sqrt{n}-1< p\sqrt{n}+s,
\end{eqnarray*}
which is exactly \eqref{p_sqrt_n+s}.
\end{proof}

We continue with a useful tool.

\begin{lemm}[Schmid \cite{schmid} and Sezer \cite{sezer}]
\label{seged}
Let $H$ be a subgroup and $N$ a normal subgroup in a finite group $G$. Then
$\beta(G) \leq \beta(N) \beta(G/N)$ and $\beta(G) \leq  |G:H| \beta(H)$.
\end{lemm}

\begin{proof}
See Schmid \cite[(3.1), (3.2)]{schmid} and Sezer \cite[Propositions~2 and~4]{sezer}.
\end{proof}

\begin{coro}
\label{pointnine}
Let $N$ be a normal subgroup of prime order $p$ in a finite group $G$. Assume that $N=C_{G}(N)$ and that $G/N$ is cyclic of order $m$ prime to $p$. Then $\beta(G) \le p  m^{0.9}$.
\end{coro}

\begin{proof}
The group $G$ is an affine Frobenius group. If $m$ is prime, then the claim follows from Lemma~\ref{q08}. For $m=4$ we have $\beta(G)\le p+6<4^{0.9}p$ by \cite[Corollary~2.9]{CzD}. If $m$ has a prime divisor $q > p/\sqrt{m}$ then first, $m<p<q\sqrt{m}$ implies $q>\sqrt{m}$. Second, $Z_p\rtimes Z_q\le G$, so by Lemma~\ref{q08} and Lemma~\ref{seged}, $\beta(G)\le \frac{m}{q}pq^{0.8}=mpq^{-0.2}<pm^{0.9}$. Finally, if $m \geq 6$ has no prime divisor larger than $p/\sqrt{m}$ then by Theorem~\ref{prootn} we have
$\beta(G)\le 2p\sqrt{m}\le pm^{0.9}$.
\end{proof}

\section{Solvable groups}
\label{solvablesection}

In this section we will give a general upper bound for $\beta(G)$ in case $G$ is a finite solvable group.

\begin{prop}
\label{nilpotent}
Let $C$ be a characteristic cyclic subgroup of maximal order in a finite nilpotent group $G$. Then $\beta(G) \leq {|C|}^{0.2} {|G|}^{0.8}$.
\end{prop}

\begin{proof}
Suppose that $G$ is a counterexample with $|G|$ minimal. By the afore-mentioned result of Noether \cite{noether}, Fleischmann \cite{fleischmann} and Fogarty \cite{fogarty}, $G$ must be non-cyclic. By Lemma \ref{seged}, $G$ must also be a $p$-group for some prime $p$. Then $G/\Phi(G)$ must be a non-cyclic elementary Abelian $p$-group where $\Phi(G)$ denotes the Frattini subgroup of $G$. By Lemma \ref{Olson}, $\beta(G/\Phi(G)) < {|G/\Phi(G)|}^{0.8}$. By minimality, there exists a characteristic cyclic subgroup $C$ in $\Phi(G)$, characteristic in $G$, such that $\beta(\Phi(G)) \leq {|C|}^{0.2} {|\Phi(G)|}^{0.8}$. We get a contradiction using Lemma \ref{seged}.
\end{proof}

We repeat the following result from the Introduction.

\begin{theo}[Domokos and Heged\H us \cite{DH} and Sezer \cite{sezer}]
\label{DomokosHegedus}
For any non-cyclic finite group $G$ we have $\beta(G) \leq \frac{3}{4}|G|$.
\end{theo}

The next bound holds for every finite solvable group, but it is slightly weaker than the one in Proposition \ref{nilpotent}.

\begin{theo}
\label{solvable}
Let $C$ be a characteristic cyclic subgroup of maximal order in a finite solvable group $G$. Then $\beta(G) \leq {|C|}^{0.1} {|G|}^{0.9}$.
\end{theo}

\begin{proof}
By Proposition \ref{nilpotent}, we may assume that $G$ is not nilpotent. Consider the Fitting subgroup $F(G)$ and the Frattini subgroup $\Phi(G)$ of $G$. Since $F(G)$ is normal in $G$, we have, by \cite[Page 269]{Huppert}, that $\Phi(F(G)) \leq \Phi(G) \leq F(G)$. Thus $F(G)/\Phi(G)$ is a product of elementary Abelian groups. The socle of the group $G/\Phi(G)$ is $F(G)/\Phi(G)$ on which $G/F(G)$ acts completely reducibly (in possibly mixed characteristic) and faithfully (see \cite[III. 4.5]{Huppert}).

Let $N$ be the product of $O_{p}(G) \cap \Phi(F(G))$ for all primes $p$ for which $O_{p}(G)$ is cyclic, together with the subgroups $O_{p}(G) \cap \Phi(F(G))$ for all primes $p$ for which $p$ divides $|F(G)/\Phi(G)|$ but $p^{2}$ does not, together with $O_{p}(G) \cap \Phi(G)$ for all primes $p$ for which $p^{2}$ divides $|F(G)/\Phi(G)|$. Clearly, $F(G)/N$ is a faithful $G/F(G)$-module (of possibly mixed characteristic) with a completely reducible, faithful quotient.

We claim that the bound in the statement of the theorem holds when $C$ is taken to be the product of the (direct) product of all cyclic Sylow subgroups of $F(G)$ and a characteristic cyclic subgroup of maximal order in $N$. By our choice of $C$ and Proposition \ref{nilpotent}, we have $\beta(N) \leq {(|C|/s)}^{0.1} {|N|}^{0.9}$, where $s$ denotes the product of the primes for which $O_{p}(G)$ is cyclic. In order to finish the proof of the theorem, it is sufficient to show that $\beta(G/N) \leq s^{0.1} {|G/N|}^{0.9}$.

This latter bound will follow from the following claim. Let $H$ be a finite solvable group with a normal subgroup $V$ that is the direct product of elementary Abelian normal subgroups of $H$. Let $\pi$ be the set of prime divisors of $|V|$ and write $V$ in the form $\times_{p \in \pi} O_{p}(V)$. Assume that $V$ is self-centralizing in $H$ and that the $H/V$-module $V$ has a completely reducible, faithful quotient module. We claim that $\beta(H) \leq s^{0.1} {|H|}^{0.9}$ where $s$ denotes the product of all primes $p$ for which $|O_{p}(V)| = p$.

We prove the claim by induction on $|\pi|$. Let $p \in \pi$. Assume that $|\pi| = 1$. If $|V| = p$ then Corollary \ref{pointnine} gives the claim. Assume that $|V| \geq p^{2}$. By a result of P\'alfy \cite{Palfy} and Wolf \cite{Wolf}, $|H/V| < {|V|}^{2.3}$. First assume that $|V|$ is different from $2^{2}$, $3^{2}$, $5^{2}$. By Lemma \ref{Olson} and Lemma \ref{seged}, $$\beta(H) < {|V|}^{0.67}|H/V| < {|H|}^{0.9}.$$ Thus assume that $|V| = 2^{2}$, $3^{2}$, or $5^{2}$. If $|H| < {|V|}^{2}$, then $$\beta(H) < {|V|}^{0.8}|H/V| < {|H|}^{0.9},$$ again by Lemmas \ref{Olson} and \ref{seged}. So assume also that $|H| \geq {|V|}^{2}$, in particular that $H/V$ is not cyclic. By Theorem \ref{DomokosHegedus}, we have $\beta(H) < \frac{3}{4} {|V|}^{0.8}|H/V| \leq {|H|}^{0.9}$, since $H$ is solvable.

Assume that $|\pi| > 1$. The group $H$ can be viewed as a subdirect product in $Y = Y_{p} \times Y_{p'}$ where $Y_{p}$ and $Y_{p'}$ are solvable groups with the following properties. There is an elementary Abelian normal $p$-subgroup $V_{p}$ in $Y_{p}$ and a direct product $V_{p'}$ of elementary Abelian normal $p'$-subgroups in $Y_{p'}$ such that both the $Y_{p}/V_{p}$-module $V_{p}$ and the $Y_{p'}/V_{p'}$-module $V_{p'}$ have a completely reducible, faithful quotient module. Let $N$ be the kernel of the projection of $H$ onto $Y_{p}$. Clearly, $N$ satisfies the inductive hypothesis with the set $\pi \setminus \{ p \}$ of primes. Thus Lemma \ref{seged} gives the bound of the claim.
\end{proof}

\section{Finite simple groups of Lie type}

The following is inherent in \cite{CzD} without being explicitly stated. We reproduce their argument with a slight modification.

\begin{lemm}
\label{czd_simple}
If $G$ is a nonsolvable finite group then $n(G)> 2.7$.
\end{lemm}

\begin{proof}
By Lemma~\ref{seged}, it is enough to prove this for minimal non-Abelian simple groups. By a theorem of Thompson \cite[Corollary~1]{thompson} these are $\PSL(3,3)$, Suzuki groups $\Sz(2^p)$, for $p>2$ prime and $\PSL(2,q)$, where $q=2^p,\,3^p$ ($p$ a prime, $p>2$ for $q=3^p$) or $q>3$ is a prime such that $q \equiv \pm 2 \pmod{5}$.

If $G\cong \Sz(2^p)$ or $G\cong\PSL(2,2^p)$, for $p>2$ then $G$ has an elementary Abelian subgroup $H\cong Z_2^3$ of index $k=|G:H|\geq 63$. So $n(G)\geq \frac{8k}{2k+3} = 4 - \frac{12}{2k+3} > 3.9$. (See the proof of \cite[Theorem 1.1 case (2a)]{CzD}.)

If $G\cong \PSL(3,3)$ or $G\cong\PSL(2,3^p)$, for $p>2$ then $G$ has an elementary Abelian subgroup $H\cong Z_3^2$ of index $k=|G:H|\geq 624$. So $n(G)\geq \frac{9k}{3k+2}=3-\frac{6}{3k+2}>2.9$. (See the proof of \cite[Theorem 1.1 case (2b)]{CzD}.)

If $G\cong \PSL(2,4)\cong A_5$ or $G\cong\PSL(2,p)$ then $G$ contains a subgroup $H\cong A_4$ of index $k=|G:H|\geq 5$. So $n(G)\geq \frac{6k}{2k+1}=3-\frac{3}{2k+1}>2.7$. (See the proof of \cite[Theorem 1.1 case (2c)]{CzD}.)
\end{proof}

This implies that if $G$ is a nonsolvable group with order less than $2.7^{40}$ then $\beta(G)< |G|/2.7 <|G|^{39/40}$. The following theorem claims this bound for every finite simple group of Lie type.

\begin{theo}
\label{lie}
Let $S$ be a finite simple group of Lie type. Then $\beta(S)\le |S|^{39/40}$, in other words, $n(S)\ge |S|^{1/40}$.
\end{theo}

\begin{proof}
The proof requires a case by case check of the $16$ families of simple groups of Lie type. In each case we find a subgroup $E\le S$ with Noether index $n(E)$ relatively large, more precisely $n(E)\ge |S|^{1/40}$ and hence $n(S)\ge n(E)\ge |S|^{1/40}$ as required.

If the rank of the group is at least $2$ then we find a non-cyclic elementary Abelian $p$-subgroup $E$ in the defining characteristic $p$ satisfying $|E|^8>|S|$. The relevant data can be found for example in \cite[Tables 3.3.1 and 2.2]{GLyS} which we summarise below. By Lemma~\ref{Olson} we have $n(S)\geq n(E)>|S|^{1/40}$ which implies our statement in this case.
However Table \ref{onetable} gives the best bounds for each type that can be obtained this way. (For notational ease $C_2(2^a)$ is used instead of $B_2(2^a)$ below. The Tits group is not in the list, but using a Sylow 2-subgroup we can easily obtain $n(S)>|S|^{0.2}$ for that $S$.)

So this method gives a better bound $\log_{|S|}n(E)\geq 0.051> 1/20$, the worst group being $S\cong\PSL(3,2)$, with $|E|=4$. 

The rank $1$ case remains. First let $p>3$ be a prime and $S=\PSL(2,p)$. Then $S$ contains a Frobenius subgroup $H\cong Z_p\rtimes Z_{(p-1)/2}$ of index $|S:H|=p+1$. By Corollary~\ref{pointnine}, we have the bound $\beta(H)\leq p (\frac{p-1}{2})^{0.9}$.
It follows by Lemma~\ref{seged} that $\beta(S)\leq (p+1)\beta(H)\leq (p+1)p(\frac{p-1}{2})^{0.9}.$ This implies
$\beta(S)<|S|^{1-1/40}$ for $p\geq 13$.

For $S\cong\PSL(2,p)$ with $p=5,\,7,\,11$ the order of the group $S$ is less than $2.7^{40}$, so the theorem holds by the remark after Lemma~\ref{czd_simple}.

\begin{table}
\caption{Elementary Abelian groups in finite simple groups of Lie type} \label{onetable}
\begin{tabular}{rccc}
\hbox{type}&\hbox{order bound}&\hbox{lower bound for }$|E|$ & \hbox{lower bound for }$\log_{|S|}n(E)$ \\
\hline
$A_m(q)$ & $q^{m^2+2m}$ & $q^{[(\frac{m+1}{2})^2]}$ & $0.11 \ (m=3,\,q=2),\, 0.051 \ (m=2,\,q=2)$ \\
${}^2A_m(q)$ & $q^{m^2+2m}$ & $q^{[\frac{m+1}{2}]^2(+1)}$ & $0.11 \ (m=3,\,q=2),\, 0.066 \ (m=2,\,q=3)$ \\
$B_m(q)$ & $q^{2m^2+m}$ & $q^{2m-1},\,q^{1+ \binom{m}{2}}$ & $0.12 \ (m=2,\,q=3)$ \\
$C_m(q)$ & $q^{2m^2+m}$ & $q^{\binom{m+1}{2}}$ & $0.15 \ (m=3,\,q=2),\, 0.15 \ (m=2,\,q=4)$ \\
$D_m(q)$ & $q^{2m^2-m}$ & $q^{\binom{m}{2}}$ & $0.11 \ (m=4,\,q=2)$ \\
${}^2D_m(q)$ & $q^{2m^2-m}$ & $q^{\binom{m}{2}},\,q^{2+{\binom{m-1}{2}}(+1)}$ & $0.11 \ (m=4,\,q=2),\,0.15 \ (m=4,\,q=3)$ \\
${}^2B_2(q)$ & $q^5$ & $q$ & $0.066 \ (q=8)$ \\
${}^3D_4(q)$ & $q^{28}$ & $q^5$ & $0.086 \ (q=2)$ \\
$G_2(q)$ & $q^{14}$ & $q^3,\,q^4$ & $0.1 \ (q=5),\,0.14 \ (q=3)$ \\
${}^2G_2(q)$ & $q^7$ & $q^2$ & $0.17 \ (q=27)$ \\
$F_4(q)$ & $q^{52}$ & $q^{11},\,q^9$ & $0.14 \ (q=2),\,0.12 \ (q=3)$ \\
${}^2F_4(q)$ & $q^{26}$ & $q^5$ & $0.14 \ (q=8)$ \\
$E_6(q)$ & $q^{78}$ & $q^{16}$ & $0.15 \ (q=2)$ \\
${}^2E_6(q)$ & $q^{78}$ & $q^{12},\,q^{13}$ & $0.11 \ (q=3),\,0.11 \ (q=2)$ \\
$E_7(q)$ & $q^{133}$ & $q^{27}$ & $0.16 \ (q=2)$ \\
$E_8(q)$ & $q^{248}$ & $q^{36}$ & $0.12 \ (q=2)$
\end{tabular}
\end{table}

Finally let $S=\PSL(2,q)$ where $q=p^f$, $p$ a prime and $f>1$. Then $S=\PSL(2,q)$ contains an elementary Abelian subgroup $E$ of order $p^f$ for which, by Lemma~\ref{Olson}, $\beta(E)=(p-1)f+1<p^{0.8f}$. Since $|S|<q^3=p^{3f}$, we have
\[n(E)=\frac{p^f}{(p-1)f+1}>p^{0.2f}>|S|^{1/15}.\]
This finishes the proof.
\end{proof}

\section{A reduction to almost simple groups}
\label{reductionsection}

We will proceed to prove the following result.

\begin{theo}
\label{mainmain}
Let $G$ be a finite group and $C$ a characteristic cyclic subgroup in $G$ of largest size. Then $\beta(G) \leq {|C|}^{\epsilon}{|G|}^{1-\epsilon}$ with $\epsilon = {(10 \log_{2} k)}^{-1}$, where $k$ denotes the maximum of $2^{10}$ and the largest degree of a non-Abelian alternating composition factor of $G$, if such exists. If $G$ is solvable, then $\beta(G) \leq {|C|}^{0.1}{|G|}^{0.9}$.
\end{theo}

The second statement of Theorem \ref{mainmain} is Theorem \ref{solvable}. The following result reduces the proof of Theorem \ref{mainmain} to a question on almost simple groups.

\begin{theo}
\label{main}
Let $G$ be a finite group. Let $\epsilon$ be a constant with $0 < \epsilon \leq 0.1$ such that $\beta(H) \leq 2^{-\epsilon} {|H|}^{1-\epsilon}$ for any (if any) almost simple group $H$ whose socle is a composition factor of $G$. Let $C$ be a characteristic cyclic subgroup of maximal order in $G$. Then $\beta(G) \leq {|C|}^{\epsilon} {|G|}^{1-\epsilon}$.
\end{theo}

Note that for any finite group $G$ the $\epsilon$ in Theorem \ref{main} can be taken to be positive by Theorem \ref{DomokosHegedus}.

\begin{proof}
Let $G$ be a counterexample to the statement of Theorem \ref{main} with $|G|$ minimal. By Theorem \ref{solvable}, $G$ cannot be solvable. Let $R$ be the solvable radical of $G$. By Theorem \ref{solvable} there exists a characteristic cyclic subgroup $C$ of $R$ (which is also characteristic in $G$) such that $\beta(R) \leq {|C|}^{\epsilon} {|R|}^{1 - \epsilon}$. If $R \not= 1$, then, by minimality, $\beta(G/R) \leq {|G/R|}^{1 - \epsilon}$, and so Lemma \ref{seged} gives a contradiction. Thus $R = 1$.

Let $S$ be the socle of $G$. This is a direct product of, say $r \geq 1$, non-Abelian simple groups. Let $K$ be the kernel of the action of $G$ on $S$. By our hypothesis on almost simple groups and by Lemma \ref{seged}, $\beta(K) \leq {|K|}^{1-\epsilon}/2^{\epsilon \cdot r}$.

Let $T = G/K$. We claim that $\beta(T) \leq 2^{\epsilon (r-1)} {|T|}^{1 - \epsilon}$. By Lemma \ref{seged} this would yield $\beta(G) \leq {|G|}^{1 - \epsilon}$, giving us a contradiction.

To prove our claim we will show that if $P$ is a permutation group of degree $n$ such that $|P| \leq |T|$, $n \leq r$, and every non-Abelian composition factor (if any) of $P$ is also a composition factor of $T$, then $\beta(P) \leq 2^{\epsilon (n-1)} {|P|}^{1 - \epsilon}$. Suppose that $P$ acts on a set $\Omega$ of size $n$. Let $P$ be a counterexample to the bound of this latter claim with $n$ minimal. Then $n > 1$. Suppose that $P$ is not transitive. Then $P$ has an orbit $\Delta$ of size, say $k$, with $k < n$. Let $B$ be the kernel of the action of $P$ on $\Delta$. Then $\beta(P/B) \leq 2^{\epsilon (k-1)} {|P/B|}^{1 - \epsilon}$ and $\beta(B) \leq 2^{\epsilon ((n-k)-1}) {|B|}^{1 - \epsilon}$. We get a contradiction using Lemma \ref{seged}. So $P$ must be transitive. Suppose that $P$ acts imprimitively on $\Omega$. Let $\Sigma$ be a (non-trivial) system of blocks with each block of size $k$ with $1 < k < n$. Let $B$ be the kernel of the action of $P$ on $\Sigma$. By minimality, $\beta(P/B) \leq 2^{\epsilon ((n/k)-1)} {|P/B|}^{1 - \epsilon}$. By minimality and Lemma \ref{seged}, we also have $\beta(B) \leq 2^{\epsilon (k-1)(n/k)} {|B|}^{1 - \epsilon}$. Again, Lemma \ref{seged} gives a contradiction. Thus $P$ must be primitive. If the solvable radical of $P$ is trivial, we get $\beta(P) \leq {|P|}^{1- \epsilon}$ by $|P| < |G|$. In fact, the same conclusion holds unless $n$ is prime and $P$ is meta-cyclic. In this latter case Corollary \ref{pointnine} gives $\beta(P) \leq n^{\epsilon} {|P|}^{1 - \epsilon}$. We get a contradiction by $n \leq 2^{n-1}$.
\end{proof}

\section{Almost simple groups}
\label{almostsimplesection}

Let $H$ be an almost simple group. In view of Theorem \ref{main} in this section we will give a bound for $\beta(H)$ of the form $2^{-\epsilon} {|H|}^{1-\epsilon}$ where $\epsilon$ is such that $0 < \epsilon \leq 0.1$. Let $S$ be the socle of $H$.

\subsection{The case when $S$ is a finite simple group of Lie type}
\label{one}
We first show that we may take $\epsilon = 0.01$. By Theorem \ref{lie},
$\beta(S) \leq {|S|}^{39/40}$. By this and Lemma \ref{seged}, we get
$$\beta(H) \leq |H:S| \cdot {|S|}^{39/40} = {|H:S|}^{0.01} \cdot {|S|}^{0.01 - (1/40)} \cdot {|H|}^{0.99}.$$
Thus it is sufficient to see that ${|H:S|}^{0.01} \cdot {|S|}^{0.01 - (1/40)} \leq 2^{-0.01}$. But this is clear since $|H:S| \leq |\mathrm{Out}(S)| < {|S|}^{1.5}/2$.

For the remainder of this subsection set $\epsilon = 0.1$. In order to prove the bound for this $\epsilon$, by the previous argument, it would be sufficient to show that $\beta(S) \leq {|S|}^{0.8}$. We claim that this holds once the Lie rank $m$ of $S$ is sufficiently large. Let $E$ be an elementary Abelian subgroup in $S$ of maximal size. By Lemma \ref{Olson} and by Table \ref{onetable}, if $m \to \infty$, we have $\log_{2}|E|/ \log_{2}\beta(E) \to \infty$. Again by Table \ref{onetable}, $\log_{2}|S|/\log_{2}|E| = 4 + o(1)$ as $m \to \infty$. Thus we have
$$\log_{2}\beta(S) \leq \log_{2}\beta(E) - \log_{2}|E| + \log_{2}|S| = (-1 + o(1)) \log_{2}|E| + \log_{2}|S| =$$
$$= (-(1/4) + o(1)) \log_{2}|S| + \log_{2}|S| = ((3/4)+o(1))\log_{2}|S| < 0.8 \log_{2}|S|,$$
as $m \to \infty$.

Let $p$ be a defining characteristic for $S$ and let $q = p^{f}$ be the size of the field of definition. Unfortunately we cannot prove the bound $\beta(H) \leq 2^{-0.1} {|H|}^{0.9}$ for all groups $H$ with $q$ large enough, but we can establish this bound in case $f$ is sufficiently large. By Table \ref{onetable}, if the Lie rank $m$ is at least $2$ then $S$ contains an elementary Abelian $p$-subgroup $E$ such that ${|E|}^{8} > |S|$. Notice that this bound also holds for $m = 1$, at least for sufficiently large groups $S$. Thus $\log_{2}|S|/\log_{2}|E| < 8$. If $f \to \infty$, then $\log_{2}|E|/ \log_{2}\beta(E) \to \infty$. In a similar way as in the previous paragraph, we obtain
$\log_{2}\beta(S) < ((7/8)+o(1)) \log_{2}|S|$, that is, $\beta(S) < {|S|}^{0.89}$, for sufficiently large $S$. Since $|H:S|$ is at most a universal constant multiple of $f$, we certainly have $|H:S| < {|S|}^{o(1)}$, as $f \to \infty$. The claim follows by Lemma \ref{seged}.

\subsection{The case when $S$ is a sporadic simple group or the Tits group}
\label{two}
In this subsection we set $\epsilon = 0.1$ and try to establish the proposed bound in as many cases as possible. Here we also complete the proof of Theorem \ref{simple}.

In this paragraph for a prime $p$ and a positive integer $k$ let $p^{k}$ denote the elementary Abelian $p$-group of rank $k$ and let $2^{1+4}$ denote a group of order $2^{5}$ with center of size $2$. By the Atlas \cite{ATLAS}, the groups $S = \jan{4}$ and $S = \co{1}$ contain a section isomorphic to $2^{12}$. Furthermore the groups $S = \co{2}$, $\co{3}$, $\mcl$, $\fis{22}$, $\fis{23}$, $\fis{24}$, $\baby$ and $\monster$ contain a section isomorphic to $2^{10}$, $3^{5}$, $3^{4}:\mat{10}$, $2^{10}$, $2^{10}$, $2^{12}$, $2^{22}$, and $2^{24}$ respectively and the group $S = \on$ contains a subgroup isomorphic to $3^{4}:2^{1+4}$. If $S$ is any of these previously listed groups, we may use Lemmas \ref{seged} and \ref{Olson} together with the estimate $\beta(\mat{10})/|\mat{10}| \leq 3/4$ in one case (see Theorem \ref{DomokosHegedus}) to obtain the bound $\beta(H) \leq 2^{-\epsilon} {|H|}^{1-\epsilon}$ with $\epsilon = 0.1$. The same estimate holds in case $S$ is the Tits group, as shown in the proof of Theorem \ref{lie}.

If $S$ is not a group treated in the previous paragraph, then $|H| < {2.7}^{40}$. Thus, by the remark after Lemma \ref{czd_simple}, we have $\beta(H) < |H|/2.7 < {|H|}^{39/40}$. This and Theorem \ref{lie} complete the proof of Theorem \ref{simple}. Notice also that for $\epsilon = 0.01$ we have ${|H|}^{39/40} < 2^{-\epsilon} {|H|}^{1-\epsilon}$.

\subsection{The case when $S$ is an alternating group}
\label{three}
Let $S = A_{k}$ be the alternating group of degree $k$ at least $5$.

Assume first that $k > 10$. Put $s=[k/4]\geq 2$. There exists an elementary Abelian $2$-subgroup $P \leq A_k$ of rank $2s$. By Lemma \ref{Olson}, we have $\beta(P)=2s+1$. By Lemma \ref{seged}, this gives
$n(S) \geq n(P) = 2^{2s}/(2s+1)$. Thus $\log_{2}(n(S)) > k \log_{2}1.11 > k/10$. This gives $\beta(H) < |H|/2^{k/10}$. Thus if $$\epsilon = \frac{k}{10 + 10 \log_{2} |H|} > \frac{1}{(10/k) + 10 (\log_{2}(k) - 1)} > \frac{1}{10 \log_{2}k},$$ then $\beta(H) < 2^{-\epsilon} {|H|}^{1-\epsilon}$.

Now let $k \leq 10$. Then $|H| < {2.7}^{16}$. By the remark after Lemma \ref{czd_simple} we have $\beta(H) < |H|/2.7 < {|H|}^{15/16}$. This is certainly less than $2^{-\epsilon} {|H|}^{1-\epsilon}$ for $\epsilon = 0.01$.

\section{Proofs of the three main results}

\begin{proof}[Proof of Theorem \ref{mainmain}]
Let $G$ be a finite group. By Theorem \ref{solvable}, we may assume that $G$ is nonsolvable. Let $H$ be an almost simple group whose socle $S$ is a composition factor of $G$. By Sections \ref{one}, \ref{two}, and \ref{three}, we see that $\beta(H) \leq 2^{-0.01} {|H|}^{0.99}$ provided that $S$ is not an alternating group of degree at least $2^{10}$. If $S$ is an alterna-ting group of degree $k$ at least $2^{10}$, then $\beta(H) \leq 2^{-\epsilon} {|H|}^{1-\epsilon}$ with $\epsilon = {(10 \log_{2}k)}^{-1}$. The result now follows from Theorem \ref{main}.
\end{proof}

\begin{proof}[Proof of Theorem \ref{mainmainindex}]
Let $G$ be a finite group with Noether index $n(G)$. Let $k$ denote the maximum of $2^{10}$ and the largest degree of a non-Abelian alternating composition factor of $G$, if such exists. Let $C$ be a characteristic cyclic subgroup in $G$ of largest possible size. Put $f = |G:C|$. By Theorem \ref{mainmain}, $\beta(G) \leq {|C|}^{\epsilon}{|G|}^{1-\epsilon}$ with $\epsilon = {(10 \log_{2} k)}^{-1}$. In other words, $n(G) \geq f^{\epsilon}$. Thus $G$ has a characteristic cyclic subgroup of index at most ${n(G)}^{10 \log_{2} k}$. If $G$ is solvable, then $\beta(G) \leq {|C|}^{0.1} {|G|}^{0.9}$ by Theorem \ref{mainmain}. In other words, $n(G) \geq f^{0.1}$ and so $f \leq {n(G)}^{10}$.
\end{proof}

\begin{proof}[Proof of Corollary \ref{corollary}]
Let $G$ be a finite group with Noether index $n(G)$. By Theorem \ref{mainmainindex} we may assume that $G$ is nonsolvable. Thus $n(G) > 2.7$ by
Lemma \ref{czd_simple}. By Theorem \ref{mainmainindex} we may also assume that $G$ has an alternating composition factor $A_{k}$ with $k \geq 2^{10}$. From Section \ref{three} we have $k < 10 \log_{2}(n(A_{k}))$. Since $n(A_{k}) \leq n(G)$ by Lemma \ref{seged}, we get $10 \leq \log_{2} k < \log_{2}10 + \log_{2} \log_{2}(n(G))$. The result now follows from Theorem \ref{mainmainindex}.
\end{proof}

\section{Questions}

We close with three questions which suggest another connection between the Noether number of a group and the Noether numbers of its special subgroups.

\begin{question}
Is it true that $\beta(S)\leq\max\{o(g)^2|g\in S\}$ for a finite simple group $S$?
\end{question}

\begin{question}
Is it true that $\beta(G) \leq \max \{ \beta(A)^{100} | A \leq G, \ A \ \mathrm{Abelian} \}$ for a finite group $G$?
\end{question}

\begin{question}
\label{mod}
Let $V$ be a finite dimensional $FG$-module for a field $F$ and finite group $G$. Is it true that $\beta(G,V) \leq \dim(V) |G:H| \beta(H,V)$ for every subgroup $H$ of $G$?
\end{question}

\centerline{\bf Acknowledgements}

The authors are grateful to M\'aty\'as Domokos for comments on an earlier version of the paper.



\end{document}